%% file: irredmaxminsemiring.tex
\documentclass[11pt,reqno]{amsart}

\usepackage{subfiles}

\input{header/packages.tex}
\input{header/journal-header}

\input{header/shortcuts.tex}

\date{\today}

\thanks{This work was supported by NSF Grants DMS1561945 and DMS1659037. We thank the participants of the 2021 Williams SMALL REU for constructive comments. Thanks also to Leo Goldmakher for translating Wirsing's paper from the original German and for helpful feedback throughout.}
\title{Irreducibility Over the Max-Min Semiring}

\author[B. Baily]{Benjamin Baily}
\email{\textcolor{blue}{\href{mailto:youremail@school.edu}{bmb2@williams.edu}}}
\address{Department of Mathematics and Statistics, Williams College, Williamstown, MA 01267}

\author[J. Dell]{Justine Dell}
\email{\textcolor{blue}{\href{mailto:jdell@haverford.edu}{jdell@haverford.edu}}}
\address{Department of Mathematics and Statistics, Haverford College, Haverford, PA 19041}

\author[H. L. Fleischmann]{Henry L. Fleischmann}
\email{\textcolor{blue}{\href{mailto:henryfl@umich.edu}{henryfl@umich.edu}}}
\address{Department of Mathematics, University of Michigan, Ann Arbor, MI 48109}

\author[F. Jackson]{Faye Jackson}
\email{\textcolor{blue}{\href{mailto:alephnil@umich.edu}{alephnil@umich.edu}}}
\address{Department of Mathematics, University of Michigan, Ann Arbor, MI 48109}

\author[S. J. Miller]{Steven J. Miller}
\email{\textcolor{blue}{\href{mailto:sjm1@williams.edu}{sjm1@williams.edu}},  \textcolor{blue}{\href{Steven.Miller.MC.96@aya.yale.edu}{Steven.Miller.MC.96@aya.yale.edu}}}
\address{Department of Mathematics and Statistics, Williams College, Williamstown, MA 01267}

\author[E. Pesikoff]{Ethan Pesikoff}
\email{\textcolor{blue}{\href{mailto:ethan.pesikoff@yale.edu}{ethan.pesikoff@yale.edu}}}
\address{Department of Mathematics,
Yale University,
New Haven, CT 06520}

\author[L. Reifenberg]{Luke Reifenberg}
\email{\textcolor{blue}{\href{mailto:lreifenb@nd.edu}{lreifenb@nd.edu}}}
\address{Department of Mathematics, University of Notre Dame, Notre Dame, IN 46556}

\date{\today}
\subjclass{11B30, 11R09, 15A80}
\begin{document}

\begin{abstract}
    For sets $A, B\subset \N$, their sumset is $A + B := \{a+b: a\in A, b\in B\}$. If we cannot write a set $C$ as $C = A+B$ with $|A|, |B|\geq 2$, then we say that $C$ is \textit{irreducible}. The question of whether a given set $C$ is irreducible arises naturally in additive combinatorics. Equivalently, we can formulate this question as one about the irreducibility of boolean polynomials, which has been discussed in previous work by K. H. Kim and F. W. Roush (2005) and Y. Shitov (2014). We prove results about the irreducibility of polynomials and power series over the max-min semiring, a natural generalization of the boolean polynomials.
    
    We use combinatorial and probabilistic methods to prove that almost all polynomials are irreducible over the max-min semiring, generalizing work of Y. Shitov (2014) and proving a 2011 conjecture by D. L. Applegate, M. Le Brun, and N. J. A. Sloane. Furthermore, we use measure-theoretic methods and apply Borel's result on normal numbers to prove that almost all power series are asymptotically irreducible over the max-min semiring. This result generalizes work of E. Wirsing (1953).
\end{abstract}

\maketitle
\tableofcontents

%%%%%%%%%%%%%%%%%%%%%%%%%%%%%%%%%%%%%%%%%%%%%%%%%%%%%%%%%%%%%%%%%%%%%%%%%%%%%%%%%%%%%%%%%%%%%%%%%%%%%%%%%%%%%%%%%%%%%%%%%%%%%%%%%%%%%%%%%%%%%%%%%%%
%%%%%%%%%%%%%%%%%%%%%%%%%%%%%%%%%%%%%%%%%%%%%%%%%%%%%%%%%%%%%%%%%%%%%%%%%%%%%%%%%%%%%%%%%%%%%%%%%%%%%%%%%%%%%%%%%%%%%%%%%%%%%%%%%%%%%%%%%%%%%%%%%%%
\section{Introduction}\label{sec:intro}
The max-min semiring is defined as $\mathcal N = (\N\cup \{\infty\}, \oplus, \otimes)$, where $a\oplus b = \max(a,b), a\otimes b = \min(a,b)$. Previous work (\cite{KR05}, \cite{Shi14}) has discussed the factorization of polynomials over the \textit{Boolean Semiring}, that is, the subsemiring $\mathcal B_2 = \{0,1\}$. In this restricted case, the two questions we hope to settle in general have already been resolved. 

\begin{defn}
Let $f\in \mathcal N[[x]]$. If $f = gh$ implies that either $g$ or $h$ is a monomial, then $f$ is \textit{irreducible}.
\end{defn}

\begin{theorem}[Shitov, 2014]\label{thm:Shitov}
As $n\to \infty$, the proportion of degree $n$ polynomials in $\mathcal B_2[x]$ which are irreducible tends to 1.
\end{theorem}

This result answers a 2005 question of K. H. Kim and F. W. Roush. Remarkably, the proof uses only elementary combinatorics and probability. 

\begin{defn}
Let $f,g\in \mathcal N[[x]]$. If $f$ and $g$ differ in only finitely many coefficients, then we say $f\sim g$.
\end{defn}

\begin{defn}
Let $f\in \mathcal N[[x]]$. If $f \sim gh$ implies that either $g$ or $h$ is a monomial, then $f$ is \textit{asymptotically irreducible}.
\end{defn}

\begin{theorem}[Wirsing, 1953]\label{thm:Wirsing}
Almost every element of $\mathcal B_2[[x]]$ is asymptotically irreducible.
\end{theorem}

This proof is measure-theoretic and builds heavily off the work of Borel, in particular the result that almost every number is \textit{normal} (\cref{defn:normal}) in base 2. Interestingly, Wirsing and Shitov phrased these results in two different settings. Wirsing in fact writes that almost every set $A\subset \N$ is asymptotically irreducible.

\begin{defn}\label{defn:sumset}
Let $A, B\subset G$ for some group $G$. Then $A+B = \{a+b: a\in A, b\in B\}$.
\end{defn}

\begin{defn}\label{defn:irreducible}
Let $S\subset \N$. If $S = A+B$ implies either $A$ or $B$ is a singleton, then $S$ is irreducible. Similarly, if $S \sim A+B$ implies $A$ or $B$ is a monomial, then $S$ is asymptotically irreducible. Here, $\sim$ denotes difference in finitely many elements just like before.
\end{defn}

We restate Shitov and Wirsing's results in these terms: almost every finite subset of $\N$ is irreducible, and almost every subset of $\N$ is asymptotically irreducible. The semiring of sets under union and set addition is isomorphic to the semiring of boolean polynomials \cite{Gro19}, hence these two formulations are equivalent. 

\ \\

Our contribution is to generalize these results to the wider setting of the max-min semiring.

\begin{theorem}\label{thm:gen-Shitov}
Fix $b$, and set $\mathcal B_b = \{0, 1, \dots, b-1\}\subset \mathcal N$. Then as $n\to \infty$, the proportion of degree $n$ polynomials in $\mathcal B_b[x]$ which are irreducible tends to 1.
\end{theorem}

\begin{theorem}\label{thm:gen-wirsing}
Almost every element of $\mathcal B_b[[x]]$ is asymptotically irreducible.
\end{theorem}

Just as products of boolean polynomials correspond to sums of sets, products of min-max polynomials correspond to sums of \textit{multisets}. For more details on this correspondence, see \cite{Gro19}. 

In Section 2, we prove \Cref{thm:gen-Shitov} by partitioning the collection of reducible polynomials in $\mathcal B_b[[x]]$ into several subcollections and bounding the size of each. We do this by applying Hoeffding's inequality and a generalization of a lemma from \cite{Shi14}. In Section 3, we prove \Cref{thm:gen-wirsing} by partitioning the set of reducible elements of $\mathcal B_b[[x]]$ into subcollections and bounding the size of each. This is done via applications of the Borel-Cantelli Lemma and the result that almost all numbers are normal in every base \cite{Wei21}.

%%%%%%%%%%%%%%%%%%%%%%%%%%%%%%%%%%%%%%%%%%%%%%%%%%%%%%%%%%%%%%%%%%%%%%%%%%%%%%%%%%%%%%%%%%%%%%%%%%%%%%%%%%%%%%%%%%%%%%%%%%%%%%%%%%%%%%%%%%%%%%%%%%%
%%%%%%%%%%%%%%%%%%%%%%%%%%%%%%%%%%%%%%%%%%%%%%%%%%%%%%%%%%%%%%%%%%%%%%%%%%%%%%%%%%%%%%%%%%%%%%%%%%%%%%%%%%%%%%%%%%%%%%%%%%%%%%%%%%%%%%%%%%%%%%%%%%%
\section{Proof of \Cref{thm:gen-Shitov}}\label{sec:Shitov}
In this section, we generalize Shitov's result to polynomials over the max-min semiring. 
% Let $\beta := b - 1$ for convenience, so our available coefficients of polynomials are $0, 1, \ldots, \beta$.
We begin with some conventions.

\begin{defn}
    Let $f\in \mathcal N[x]$. Then $\abs{f}$ denotes the number of nonzero coefficients of $f$.
\end{defn}

\begin{defn}[\cite{ALS11}]
    A \textit{digit map} is a nondecreasing function $\N\to \N$. If $d$ is a digit map and $f = a_0 \oplus a_1x \oplus a_2x^2\oplus \cdots\in \mathcal N[[x]]$, then we let $d(f) = d(a_0) \oplus d(a_1)x \oplus d(a_2)x^2\oplus \cdots$.
\end{defn}

\begin{prop}[\cite{ALS11}]
    If $d$ is a digit map, then $d: \mathcal N[[x]]\to \mathcal N[[x]]$ is a semiring homomorphism. In particular, if $f = gh$ is a nontrivial factorization and $d(1) \geq 1$, then $d(f)=d(g)d(h)$ is a nontrivial factorization of $d(f)$.
\end{prop}

This key idea provides a powerful framework for our proof, allowing us to partially reduce the problem of factoring a polynomial over $\mathcal B_b$ to factoring one over $\mathcal B_2$. 

\begin{defn}\label{defn:canonical-digit-maps}
    We define the digit maps maps $s_i$ for each $i\in \Z^+$. 
    \begin{align}
        s_i(n)\ :=\ \begin{cases}
        0 & n < i\\
        i & n \geq i.
        \end{cases}
    \end{align}
    For $f\in \mathcal N[[x]]$, we additionally define $f_i = s_1(s_i(f))$. These polynomials, which we refer to as the ``$i$-level support of $f$'', are indicator functions for where the coefficients of $f$ are at least $i$.
\end{defn}

Finally, to conclude our setup, we use the following convention for referencing the coefficients of polynomials.

\begin{defn}
Throughout the remainder of this paper, let 
    \begin{align*}
        f\ =\ \bigoplus_{k=0}^\infty \alpha_kx^k,\ \,\ \, g\ =\ \bigoplus_{k=0}^{\infty} \beta_kx^k,\ \,\ \, h\ =\ \bigoplus_{k=0}^\infty \gamma_kx^k,\ \,\ \, \sigma = \bigoplus_{k=0}^\infty \delta_kx^k.\\
    \end{align*}
Additionally, set $\alpha'_i = \alpha_i\otimes 1$ and similarly for each other coefficient. This way we have, for instance:
\[
f_1\ =\ \bigoplus_{k=0}^\infty \alpha'_kx^k,\ \,\ \, g_1\ =\ \bigoplus_{k=0}^{\infty} \beta'_kx^k,\ \,\ \, h_1\ =\ \bigoplus_{k=0}^\infty \gamma'_kx^k,\ \,\ \, \sigma_1 = \bigoplus_{k=0}^\infty \delta'_kx^k.
\]
\end{defn}

In the proof of another lemma, Shitov shows the following statement, which will be of great use to us.

\begin{cor}[\cite{Shi14}]\label{cor:close-factors}
 For any $d>0$, the number of pairs of boolean polynomials $(f,g)$ satisfying the following conditions is at most $n^{2d+1}2^{(k,n)}$.
\begin{enumerate}
    \item The constant terms of $f,g$ are nonzero;
    \item $\deg f = k > 0, \deg g = n-k$;
    \item $\abs{f\otimes g} \leq \abs{f} + \abs{g} + d$.
\end{enumerate}
\end{cor}

Due to our slightly different convention about irreducibility, we require a version of this lemma when the condition (1) is not necessarily satisfied. We now address the opposing conventions of irreducibility. 

Our definition of irreducible is slightly broader than the traditional notion of irreducibility over a semiring ($f = gh\implies g$ is a unit). The advantages of our definition are as follows. 

\begin{itemize}
    \item If $f\in \mathcal B_b[[x]]$ is irreducible, then $f$ is also irreducible over $\mathcal N[[x]]$. In contrast, $f = (b-1)\otimes f$ is trivial factorization over $B_b[[x]]$, but a nontrivial factorization over $\mathcal N[[x]]$ as $b-1$ is no longer a unit. Thus, despite introducing no new factors, $f$ is now reducible. 
    
    \item The Taylor series $f$ such that $g\mapsto f\otimes g$ is an injective endomorphism of any semiring $\mathcal B_b[[x]]$ are precisely the monomials. Thus, despite not having multiplicative inverses, multiplication by monomials is invertible in the sense that we have cancellation.
    
    \item This definition lines up with the additive combinatorics. For instance, the set $\{1,2,4\}$ is additively indecomposible, but the polynomial $x \oplus x^2\oplus x^4$ is reducible as $x(1\oplus x\oplus x^3)$ under previous authors' definitions. 
\end{itemize}

We now generalize \Cref{cor:close-factors} to our setting.

\begin{lemma}\label{lemma:close-factors-2}
    The number of pairs boolean polynomials $(f,g)$ satisfying the following conditions is at most $n^{2d+2}2^k$ for any $d>0$.
    \begin{enumerate}
        \item The constant term of $f$ is nonzero;
        \item $\deg f = k>0, \deg g = n-k$;
        \item $\abs{f\otimes g} \leq \abs{f} + \abs{g} + d$.
    \end{enumerate}
\end{lemma}

\begin{proof}
     Write $g = x^j\otimes (1 + \cdots + x^{n-k-j})$ and define $\bar g$ by $g = x^j\otimes \bar g$. Then clearly $\abs{f\otimes g} = \abs{f\otimes \bar g}$ and $\abs{g} = \abs{\bar g}$. By \Cref{cor:close-factors}, there are at most $n^{2d+1}2^{(k,n-j)}$ pairs $(f,\bar g)$ satisfying the hypotheses of the corollary. Since there are at most $n$ choices for $j$, the number of pairs $(f,g)$ satisfying the hypotheses of this lemma is at most $\sum_{j=0}^{n-1} n^{2d+1}2^{(k,n-j)} \leq n^{2d+2}2^k$. 
\end{proof}

The final ingredient for our proof is Hoeffding's inequality: a probabilistic lemma which Shitov used, in conjunction with \Cref{cor:close-factors}, to prove \Cref{thm:Shitov}.

\begin{prop}[Hoeffding's Inequality]
Let $X_n$ be a sum of $n$ independent Bernoulli random variables $X$ with $\E[X] = p$. Then $P(\abs{X_n-np} > \epsilon n) \leq 2e^{-2\epsilon^2 n}$.
\end{prop}

\begin{proof}
See \cite{Hoe63}, Theorem 2.
\end{proof}

When we choose a degree $n-1$ polynomial $f$ at random from $\mathcal B_b[x]$, the quantity $\abs{f_i}$ is a sum of $n$ independent Bernoulli random variables $Z_i$ with $\mathbb E[Z_i] = \frac{b-i}{b}$. As a consequence, if $f$ is a degree $n-1$ polynomial chosen randomly from $\mathcal B_b[x]$, then
\begin{equation}\label{eqn:hoeffding-support}
 P\left(\abs{\abs{f_i}- \frac{(b-i)n}{b}} > \epsilon n\right) \leq 2e^{-2\epsilon^2n}.
\end{equation}

\begin{defn}
If $f, g$ are nonnegative real-valued functions and there exists a constant $c>0$ such that $f \leq cg$, then we write $f\lesssim g$.
\end{defn}

We now prove a quantitative version of \Cref{thm:gen-Shitov}. 

\begin{prop}\label{prop:quantitative-gen-Shitov}
Let $b>1$ and $a = \lfloor b/2\rfloor$ and let $\Sigma_{b,n}$ denote the set of reducible degree $n-1$ polynomials in $\mathcal B_b[x]$. Then for any $d,v>0$ we have
\begin{align*}
    \abs{\Sigma_{b,n}}\ \lesssim b^n\ \left (ne^{-d^2/4(n+1)} + vn^{2d+1}2^vb^{-n} + n^22^{-v} + n^{2d+3}2^{\frac{d}{2}-\frac{n}{3}}\right).
\end{align*}
\end{prop}

\begin{proof}
     We partition $\Sigma_{b,n}$ into 7 sets $\Sigma_{b,n}\subset E^1_n(d,v) \cup \cdots \cup E^7_n(d,v)$. Our proposition follows from the bound $\abs{\Sigma_{b,n}} \leq \abs{E^1_n(d,v)} + \cdots + \abs{E^7_n(d,v)}$.
     
\label{defn:5-set-partition}
     We now detail the partition. Though we will not write this after each set, we stipulate that $h\in E^i_n(d,v)$ only if $h\notin E^j_n(d,v)$ for any $j<i$.
     
         \begin{itemize}
             \item[] $E^1_n(d,v)$ is the set of polynomials $h$ such that $\abs{\abs{h_1} - \frac{(b-1)n}{b}} > \frac{d}{2}$.
             \item[] $E^2_n(d,v)$ is those $h = f\otimes g$ such that $\abs{\abs{f_1} + \abs{g_1} - \frac{(b-1)(n+1)}{b}} > \frac{d}{2}$.
             \item[] $E^3_n(d,v)$ is those $h$ such that $\abs{\abs{h_a} - \frac{(b-a)n}{b}} > \frac{d}{2}$.
             \item[] $E^4_n(d,v)$ is those $h = f\otimes g$ such that $\abs{\abs{f_a} + \abs{g_a} - \frac{(b-a)(n+1)}{b}} > \frac{d}{2}$.\\
             \item[] The size of each of these sets can be bounded using \Cref{eqn:hoeffding-support}. We start by considering these sets in order to control the size of the supports of the polynomials in the remaining sets. In particular, we want $|h_i| \leq |f_i| + |g_i| + d$ for $i = 1, a$. This is so that when $h_i = f_i\otimes g_i$ is a nontrivial factorization, the hypotheses of \Cref{lemma:close-factors-2} apply and we can conclude that there are few possible pairs $(f, g)$.\\
             \item[] $E^5_n(d,v)$ is those $h = f\otimes g$ with $\deg f\leq v$.
             \item[] $E^6_n(d,v)$ is those $h = f\otimes g$ with $\abs{f_a} \leq 1$ or $\abs{g_a}\leq 1$. \\
             \item[]The set $E^6_n(d,v)$ contains polynomials $h = f \otimes g$ where $h_a = f_a \otimes g_a$ is a trivial factorization. However, if $\deg f\leq v$, then $h \in E^5_n(d,v)$, and thus not in $E^6_n(d,v)$. Using this fact allows us to achieve the necessary upper bound on the size of $E^6_n(d,v)$.
             \\
             \item[] $E^7_n(d,v)$ is all remaining reducible degree $n-1$ polynomials $h$. Once the first 6 sets are considered, every remaining reducible polynomial $h = f\otimes g$ satisfies $|h_a| \leq |f_a| + |g_a| + d$, and $h_a = f_a\otimes g_a$ is a nontrivial factorization. Thus, we are able to apply \Cref{lemma:close-factors-2} to conclude that the number of remaining reducible polynomials is small.
         \end{itemize}

     Now, we are ready to bound the size of each of these sets.

     \begin{enumerate}
         \item By \Cref{eqn:hoeffding-support} with $\epsilon = \frac{d}{2n}$, we obtain
         \[\abs{E^1_n(d,v)}, \abs{E^3_n(d,v)} \leq 2e^{-d^2/4n}b^n\lesssim e^{-d^2/4n}b^n \leq e^{-d^2/4(n+1)}b^n.\]
         \item[]
         % comment the extra items out eventually, or reformat this proof. This is just to make it easier to read.
         \item Each pair $(f,g)$ corresponds to only one choice of $h$, thus it suffices to bound the number of pairs $(f,g)$. If we fix $\deg f = k$, then we must have $\deg g = n-k-1$ as $\deg (f\otimes g) = n-1$. The set of pairs $(f,g)\in (\mathcal B_b[x])^2$ such that $\deg f = k, \deg g = n-k-1$ is in bijection with the set of degree $n$ polynomials of $\mathcal B_b[x]$, with the bijection given below:
         \begin{align*}
            \phi(f,g) & \ =\  f \oplus \left(((b-1) x^{k+1})\otimes g\right)\\ 
            \phi^{-1}(h) &\ = \ \left(\bigoplus_{j=0}^k \gamma_jx^j, \bigoplus_{j=k+1}^n \gamma_jx^{j-k-1}\right).
         \end{align*}
         Moreover, $\abs{f_i} + \abs{g_i} = \abs{(\phi(f,g))_i}$. 
         Thus, choosing $\epsilon = \frac{d}{2n+2}$ and applying \Cref{eqn:hoeffding-support}, we obtain that there are at most $2e^{-d^2/4(n+1)}b^{n+1}$ such pairs $(f,g)$. Since there are $n/2$ choices for $\deg f$, we use this bound for each choice and obtain
         \[
         \abs{E^2_n(d,v)}, \abs{E^4_n(d,v)}\ \leq\ ne^{-d^2/4(n+1)}b^{n+1} \ \lesssim\ ne^{-d^2/4(n+1)}b^n.
         \]
         \item[]
         \item Let $h = f\otimes g\in E^5_n(d,v)$. Since $h\notin E^1_n(d,v)\cup E^3_n(d,v)$, we have $\abs{h_1} \leq \frac{(b-1)n}{b} + \frac{d}{2}, \abs{f_1} + \abs{g_1} \geq \frac{(b-1)(n+1)}{b} - \frac{d}{2}$. Thus $\abs{h_1} \leq \abs{f_1} + \abs{g_1} + d$. If $\abs{f_1} \leq 1$ or $\abs{g_1}\leq 1$, then $f$ or $g$ is a monomial in contradiction to the assumption that $f\otimes g$ is a nontrivial factorization, hence $(f_1,g_1)$ satisfy every hypothesis of \Cref{lemma:close-factors-2}. We apply this lemma once for each choice of $1\leq \deg f \leq v$, and conclude 
         \[
            \abs{E^5_n(d,v)}\ \leq\ \sum_{\deg f = 1}^vn^{2d+1}2^{\deg f}\ \leq\ vn^{2d+1}2^v.
        \]
         \item[]
         \item Suppose $\abs{f_a}\leq 1$. Then fix $\deg f = k$. Since $h\notin E^5_n(d,v)$, we can assume $k>v$. Then there are $(k+1)(b-a)(a-1)^{k}$ choices\footnote{Pick the index of the coefficient to be at least $a$, then pick its value, then pick the remaining coefficients from $\{0,\cdots, a-1\}$.} for $f$ and $b^{n-k}$ choices for $g$, hence there are $\leq (k+1)(a-1)^kb^{n-k+1}$ pairs $(f,g)$. There are at most $n$ choices for $k$, hence 
         \begin{align*}
            \abs{E^6_n(d,v)} &\ \leq\ \sum_{k=v}^n (k+1)(a-1)^kb^{n-k+1}\ \leq\ n(n+1) (a-1)^vb^{n-v+1} \\
            &\ \lesssim\ n^2 (a-1)^vb^{n-v}\ \leq\   
            n^2\left(\frac{b}{2}\right)^vb^{n-v} \ \leq  n^22^{-v}b^n.
         \end{align*}
         If instead $\abs{g_a}\leq 1$, then we have that $\deg(g) \geq \deg(f)\geq v$. By symmetry, there are at most twice as many pairs with either $|f_a| \leq 1$ or $|g_a| \leq 1$ as there are with $|f_a|\leq 1$. This doubles the size of our upper bound, but this is only a constant factor.
         \item[]
         \item Let $h = f\otimes g$. Since $h\notin E^3_n(d,v)\cup E^4_n(d,v)$, we have $\abs{h_a} < \frac{(b-a)n}{b} + \frac{d}{2}$ and $\abs{f_a} + \abs{g_a} > \frac{(b-a)(n+1)}{b}$, hence $\abs{h_a}\leq \abs{f_a} + \abs{g_a} + d$. Moreover, as $h\notin E^6_n(d,v)$, neither $f_a$ nor $g_a$ is a monomial and thus the pair $(f_a, g_a)$ is a nontrivial factorization of $h_a$ and satisfies the hypotheses of \Cref{lemma:close-factors-2}. Thus, using the fact that $(\deg f_a, n-1) \leq \frac{n-1}{2}\leq \frac{n}{2}$ for $1\leq \deg f\leq n-2$, the number of possible choices for $h_a$ is at most
         \[\sum_{\deg f_a = 1}^{n-2} n^{2d+2}2^{(\deg f_a,n-1)} \ \leq\ n^{2d+3}2^{\frac{n}{2}}.\]
         Once $h_a$ is known, if $\abs{h_a} = k$, there are $a^{n-k}(b-a)^k$ choices for $h$. This is because each 0 coefficient of $h_a$ can correspond to any coefficient in $\{0,\ldots, a-1\}$, and any 1 corresponds to a coefficient in $\{a, \ldots, b-1\}$. Since $h\notin E^3_n(d,v)$, we can say $k\leq \frac{(b-a)n}{b} + \frac{d}{2}$, a quantity which we denote by $s$ to clean up our expressions. Recalling that $a := \floor{b/2}$, we have $a \leq b/2\leq (b-a)$, with equality of all terms when $b$ is even. This gives the following upper bound: 
         \begin{align*}
           a^{n-k}(b-a)^k&\ \leq\ a^{n-s}(b-a)^s
         \\&\ \leq\  a^{\frac{a n}{b}}(b-a)^{\frac{(b-a)n}{b}}
        \left(\frac{b-a}{a}\right)^{\frac{d}{2}}
        \\&\ \leq\ a^{\frac{n}{2}}(b-a)^{\frac{n}{2}}\left(\frac{b-a}{a}\right)^{\left(\frac{b-a}{b}-\frac{1}{2}\right)n}
        \left(\frac{b-a}{a}\right)^{\frac{d}{2}}
        \\&\ \leq\ 
        \left(\frac{b}{2}\right)^n\left(\frac{b-a}{a}\right)^{\left(\frac{b-a}{b}-\frac{1}{2}\right)n} \left(\frac{b-a}{a}\right)^{\frac{d}{2}}.
         \end{align*}
        For $b\geq 2$, we have the bounds $1\leq \frac{b-a}{a}\leq 2$ and $0\leq \frac{b-a}{b} - \frac{1}{2} \leq \frac{1}{6}$, both of which are achieved when $b = 3$. Moreover, for $b = 2$, we have $\left(\frac{b-a}{b}^{\frac{b-a}{b}-\frac{1}{2}}\right) = 1$, thus for any $b\geq 2$ we have $\left(\frac{b-a}{b}\right)^{(\frac{b-a}{b}-\frac{1}{2})n}\leq 2^{\frac{n}{6}}$. Altogether, this yields:
        \begin{align*}
            \abs{E^7_n(d,v)} &\ \leq\ n^{2d+3}2^{\frac{n}{2}}\left(\frac{b}{2}\right)^n \left(\frac{b-a}{a}\right)^{\frac{d}{2}}\left(\frac{b-a}{b}\right)^{(\frac{b-a}{b}-\frac{1}{2})n} \\
            &\ \leq\  n^{2d+3}2^{\frac{n}{2}}\left(\frac{b}{2}\right)^n 2^{\frac{n}{6}}2^{\frac{d}{2}} \ \leq\ n^{2d+3}2^{\frac{d}{2}-\frac{n}{3}}b^n.
        \end{align*}
     \end{enumerate}
\end{proof}

% Now, we choose $d = \log n\sqrt{n},v = (\log_2 n)^2$, and win! We will show that the contribution from each set is negligible with respect to $b^n$. %better: v = 3log_2(n)
% \begin{itemize}
%     \item[$\abs{A_n}/b^n$]$\leq 2ne^{-(\log n)^2/c_1}\to 0$
%     \item[$\abs{B_n}/b^n$]$\leq 2ne^{-(\log n)^2/c_2}\to 0$
%     \item[$\log (\abs{C_n}/b^n)$]$\leq 2\log n+2\sqrt{n}(\log n)^2+\sqrt{n}\log n+2 - (n/b-\log n\sqrt{n})\log b\sim -n\log b/b \to -\infty$, hence $\abs{C_n}/b^n\to 0$.
%     \item[$\abs{D_n}/b^n$]$\leq n^22^{-\sqrt{n}\log n}e^{-(\log n)^4/c_3n}\to n^22^{-\sqrt{n}\log n}\to 0$.
%     \item[$\log(\abs{E_n}/b^n)$]$\leq (2\sqrt{n}\log n + 5)\log n - n/2\log 2 \to -\infty$, hence $\abs{E_n}/b^n\to 0$.
% \end{itemize}
With the right choice of $d,v$, this gives us a proof of \Cref{thm:gen-Shitov}.

\begin{proof}
     Our goal is to show that $\frac{\abs{\Sigma_{b,n}}}{b^n}\to 0$, from which the result follows. Set $d = 2\sqrt{n+1}\log n$ and $v = 3\log_2 n$ then apply \Cref{prop:quantitative-gen-Shitov}. We show that each summand of the upper bound on $\frac{\abs{\Sigma_{b,n}}}{b^n}$ vanishes. 
     \begin{enumerate}
         \item We have $ne^{-d^2/4(n+1)} = ne^{-(\log n)^2} = n^{1-\log n}$, which vanishes as $n\to \infty$.
         \item[]
         \item We have $\log (vn^{2d+1}2^vb^{-n})  = \log v + (4\sqrt{n+1}\log n + 2)\log n + 3\log_2 n\log 2 - n\log b$. Each summand of this expression is sub-linear except for the one which is negative, therefore this diverges to $-\infty$. It follows that
         \[(vn^{2d+1}2^vb^{-n}) = e^{\log (vn^{2d+1}2^vb^{-n})} \to 0.\]
         \item[]
         \item We have $n^22^{-v} = n^22^{-3\log_2 n} = n^{-1} \to 0$.
         \item[]
         \item We have $\log \left(n^{2d+3}2^{\frac{d}{2}-\frac{n}{3}}\right) = 4\sqrt{n+1}\log n + (\sqrt{n+1}\log n - \frac{n}{3})\log 2$. By the same reasoning as the bound on the second summand, we conclude $n^{2d+3}2^{\frac{d}{2}-\frac{n}{3}}\to 0$.
     \end{enumerate}
\end{proof}

This result in hand, we are now prepared to state and prove the conjecture of Applegate, LeBrun, and Sloane \cite{ALS11}. Their conjecture refers to prime elements of the semiring, which we define below, and is in some ways a more natural definition.

\begin{defn}\label{defn:prime-polynomial}
    A polynomial $h\in \mathcal B_b[x]$ is \textit{prime} if $h = f\otimes g$ implies either $f,g = b-1$. 
\end{defn}

\begin{conjecture}\label{conj:lunar-prime-conj}
Let $\pi_b(n)$ denote the number of degree $n-1$ prime polynomials of $\mathcal B_b[x]$. Then $\pi_b(n)\sim (b-1)^2b^{n-2}$.
\end{conjecture}

Motivating their conjecture, Applegate et al. observed that only certain polynomials can be prime.

\begin{defn}\label{defn:prime-candidate}
    A \textit{prime candidate} of $\mathcal B_b[x]$ is a polynomial with nonzero constant term and maximum coefficient $b-1$.
\end{defn}

It is easy enough to see that a polynomial is prime only if it is a prime candidate. If $h = a_jx^j \oplus \cdots \oplus a_{n-1}x^{n-1}$ for $j>1$, then $h = (b-1)x^j\otimes (a_j \oplus \cdots \oplus a_{n-1}x^{n-j-1})$ which is a nontrivial factorization in their convention. Moreover, if $c<b-1$ is the maximum coefficient of $h$, then $h = c\otimes h$. 

They showed that the number of prime candidates is asymptotic to $(b-1)^2b^{n-2}$, and from their data\footnote{OEIS sequences \href{https://oeis.org/A169912}{(A169912)}, \href{https://oeis.org/A087636}{(A087636)} show the number of prime elements of $\mathcal B_2[x], \mathcal B_{10}[x]$ of each degree $n$.}, as $k\to \infty$, almost all prime candidates are in fact prime. As evidence for this fact, Applegate et al. produced the following lower bound:
\begin{equation*}
    (b-1)^{n-2} + 2(b-2)^{n-2} + \cdots \leq \pi_b(n).
\end{equation*}
Moreover, they observed the following, which we will re-prove here.
\begin{lemma}[\cite{ALS11}]\label{lemma:irreducible-prime-candidate}
    An irreducible prime candidate is prime.
\end{lemma}
\begin{proof}
     If $h$ is irreducible, then $h = fg$ implies either $f,g$ is a monomial, without loss of generality, $f$ is. Since the constant term of $h$ is nonzero, we must have that $f$ is a constant. Since the maximum coefficient of $h$ is $b-1$, we must also have that $f = b-1$, thus $h$ is prime.
\end{proof}
With this lemma, \Cref{conj:lunar-prime-conj} is a simple corollary of \Cref{thm:gen-Shitov}.

\begin{proof}
     The proportion of degree $n-1$ prime candidates of $\mathcal B_b[x]$ which are irreducible is at most a quantity which vanishes as $n\to\infty$:
     \begin{equation}\label{eqn:vanishing-candidates}
         \frac{\abs{\Sigma_{b,n}}}{(b-1)^2b^{n-2}} \lesssim  \frac{\abs{\Sigma_{b,n}}}{b^n} \to 0.
     \end{equation}It follows that almost all prime candidates are prime.
\end{proof}

%%%%%%%%%%%%%%%%%%%%%%%%%%%%%%%%%%%%%%%%%%%%%%%%%%%%%%%%%%%%%%%%%%%%%%%%%%%%%%%%%%%%%%%%%%%%%%%%%%%%%%%%%%%%%%%%%%%%%%%%%%%%%%%%%%%%%%%%%%%%%%%%%%%
%%%%%%%%%%%%%%%%%%%%%%%%%%%%%%%%%%%%%%%%%%%%%%%%%%%%%%%%%%%%%%%%%%%%%%%%%%%%%%%%%%%%%%%%%%%%%%%%%%%%%%%%%%%%%%%%%%%%%%%%%%%%%%%%%%%%%%%%%%%%%%%%%%%
\section{Proof of \Cref{thm:gen-wirsing}}
Before we prove this, we first must clarify what we mean by ``almost all.'' It turns out, there is a very natural measure to associate to the set $\mathcal B_b[[x]]$.

\begin{defn}
    To each element of $\mathcal B_b[[x]]$ we associate a real number in $[0, b]$, given by
    \begin{equation}
\rho_b\left(\bigoplus_{k=0}^\infty a_kx^k\right) \ := \ \sum_{k=0}^{\infty}a_kb^{-n}.
\end{equation}
\end{defn}

In other words, each power series corresponds to a string of digits in $[0,1,\dots,b-1]$, which we can interpret as the base-$b$ expansion of a number. This allows us to define a probability measure $m$ on $\mathcal B_b[[x]]$.

\begin{defn}\label{defn:measure}
    For a set $A \subset \mathcal B_b[[x]]$ such that $\rho_b(A)$ is a measurable subset of $\R$, let $m(A) = b^{-1}{\mathcal L(\rho_b(A))}$, where $\mathcal L$ denotes the Lebesgue measure.
\end{defn}

This reframing allows us to ask and answer questions about these polynomials measure-theoretically.  For example, we will use Borel's theorem that every number is normal, regardless of base \cite{Wei21}. 

We deduce \Cref{thm:gen-wirsing} from a second theorem.

\begin{theorem} \label{thm:gen-wirsing-measure}
    Let $\mathcal C^b\subset \mathcal B_b[[x]]$ denote the set of reducible polynomials. Then $m(\mathcal C^b) = 0$.
\end{theorem}

We show first how \Cref{thm:gen-wirsing} follows from \Cref{thm:gen-wirsing-measure}.

\begin{proof}
    For $f\in \mathcal B_b[[x]]$, let $[f]$ denote the set of all $g$ such that $f\sim g$. The set of asymptotically reducible $f$ is precisely the set $[\mathcal C^b]$. Fix a natural number $n$, and notice the set of functions which differ in exactly $n$ coefficients from some element of $C^b$ has measure 0. Thus, $[C^b]$ is a countable union of measure 0 sets, hence it has measure 0 and almost all power series over $\mathcal B_b[[x]]$ are irreducible.
\end{proof}

We now prove \Cref{thm:gen-wirsing-measure}. Our proof parallels Wirsing's original argument to a great extent, but as the authors are not aware of an English translation of Wirsing's result \cite{Wir53}, we reproduce it here for the sake of completeness.

\begin{defn}\label{defn:A(x)}
    For $n\in\N$ and $f = \bigoplus_{k=0}^\infty a_kx^k \in \mathcal N[[x]]$, define $f(n) := \bigoplus_{k=0}^n a_kx^k \in \mathcal N[x]$. 
\end{defn}

First, partition $\mathcal{C}^b$ into three sets $T^b_1,T^b_2,T^b_3$:
\begin{align*}
    T^b_1 & \ :=\  \left\{h:h=f\otimes g \textrm{ with } 2\leq \abs{g_1}<\infty\right\}\\
    T^b_2 &\ :=\  \left\{h:h=f\otimes g \textrm{ with } \liminf_{n\rightarrow\infty}\frac{\abs{f_1(n)}+\abs{g_1(n)}}{n}<\frac{1}{5} \textrm{ and } \abs{f_1}=\infty=\abs{g_1}\right\}\\
    T^b_3 & \ :=\  \left\{ h:h=f\otimes g \textrm{ with } \liminf_{n\rightarrow\infty}\frac{\abs{f_1(n)}+\abs{g_1(n)}}{n}\geq\frac{1}{5} \textrm{ and } \abs{f_1}=\infty=\abs{g_1} \right\}.
    \end{align*}
Since $T^b_1 \cup T^b_2 \cup T^b_2 = C^b$, it suffices to show that $\mathcal L(T^b_1) = \mathcal L(T^b_2) = \mathcal L(T^b_3) = 0$.  In proving that the measures of $T^b_1$ and $T^b_3$ are 0, we rely extensively on the following idea.

\begin{defn}\label{defn:normal}
    A number $\lambda\in\R$ is \textit{normal} in base $b$ if the base $b$ representation of $\lambda$ contains an equal proportion of each finite sequence of digits base $b$.  That is, if for all positive integers $n$, all possible strings of $n$ digits  have density $b^{-n}$ in the base $b$ representation. 
    
    More formally, let $s = (\delta_1, \dots, \delta_k)$ be a string of digits in $\{0, \dots, b-1\}$. Fix a real number $\lambda$ and let $N_\lambda(n, s)$ denote the number of occurences of the string $s$ in the first $n$ digits of the base-$b$ expansion of $\lambda$. Then the following holds:
    \[
    \lim_{n\to\infty} \frac{N_\lambda(n,s)}{n}\ =\ b^{-k}.
    \]
    An equivalent formulation of this is the following: let $Z \subset \{0, \dots, b-1\}^k$ and let $N_\lambda(n,Z) = \sum_{s\in Z} N(n,s)$. Then
    \begin{equation}\label{eqn:normality-expression}
        \lim_{n\to\infty}\frac{N(n,Z)}{n} = |Z|b^{-k}.
    \end{equation}
\end{defn}

\begin{theorem} \label{thm:normal}
(Borel, 1909 for base 2; Wirsing, 1953 for the general case) 
    For any $b\geq2$, almost every $\lambda\in\R$ is normal base $b$. Consequently, almost every $\lambda \in \R$ is absolutely normal, that is, normal in every base.
\end{theorem}

We now state an important lemma with an elementary proof.

\begin{lemma}\label{lemma:summand-inequality}
    If $h = f\otimes g$, then $\bigoplus_{k=0}^n \alpha_k\otimes \beta_{n - k} = \gamma_n$.
\end{lemma}

\begin{proof} To elucidate this fact, all we need to do is rewrite the product $f\otimes g$:
\[
f\otimes g\ =\ \bigoplus_{i=0}^\infty \alpha_ix^i \bigoplus_{j=0}^\infty \beta_jx^j\ =\ \bigoplus_{n=0}^\infty \left(\bigoplus_{k=0}^n \alpha_k\otimes \beta_{n-k}\right)x^n\ =\  \bigoplus_{n=0}^\infty \gamma_nx^n.
\]
\end{proof}

\begin{lemma} \label{lem: t1} We have
    $m(T^b_1) = 0$.
\end{lemma}

\begin{proof}
    We show that no element of $T^b_1$ is normal, whence the result follows.  Specifically, we claim that the following sequence of digits can never occur in $\rho_b(h)$ for any $h = f\otimes g\in T^b_1$:
    \begin{align}\label{eqn:digit-string}
        \underbrace{00\dots0}_{\deg g+1}1 \underbrace{00\dots0}_{\deg g+1}.
    \end{align}
    Let $f_1 = \bigoplus_{k=0}^\infty \alpha_kx^k, g_1 = \bigoplus_{k=0}^{\deg g_1} \beta_kx^k, h_1 = \bigoplus_{k=0}^\infty \gamma_kx^k$. We can write
    \[
    h_1\ =\ g_1\otimes f_1\ =\ \bigoplus_{k=0}^{\deg g} \beta_kx^k\otimes f_1.
    \]
    If $\gamma_k = 1$, then by \Cref{lemma:summand-inequality} there exist $i, j$ such that $\alpha_i = \beta_j = 1$ and $i+j = k$. Since $g_1$ is not a monomial, there exists another index $j'\neq j$ such that $\beta_{j'} = 1$. Then by \Cref{lemma:summand-inequality}: $1\leq \gamma_{i+j'}$ and $\gamma_{i+j'}' = 1$. The gap between the two indices $i+j, i+j'$ is at most $\deg g_1$ (but either index can come first), thus $\rho_b(h_1)$ does not have a ``1'' without another ``1'' at most $\deg g$ indices away. Thus the string \Cref{eqn:digit-string} does not occur in $\rho_b(h_1)$. 
\end{proof}

\begin{lemma} \label{lem: t2}
   We have $m(T^b_2) = 0$.
\end{lemma}

\begin{proof}
    We begin by defining a finite counterpart to $T^b_2$:
    \begin{equation*}
        T^b_2(n) \ :=\ \left\{\rho_b(h): h=f\otimes g:\frac{\abs{f_1(n)}+\abs{g_1(n)}}{n}<\frac{1}{5} \textrm{ and } \abs{f_1}=\infty=\abs{g_1}\right\}.
    \end{equation*}
    Notice that
    \begin{equation*}                                   
        T^b_2 \subseteq \limsup (\{T^b_2(n)\})\ =\ \bigcap_{N\geq1}\bigcup_{n\geq N} T^b_2(n).
    \end{equation*}
    By the Borel-Cantelli Lemma, we know that if $$\sum_{n=1}^{\infty}m(T^b_2(n))<\infty,$$ then 
    $$m\left(\limsup_{n \to \infty}(T^b_2(n))\right) = m(T^b_2) = 0.$$
    As such, it suffices to show that $\sum_{n=1}^{\infty}m(T^b_2(n))<\infty$.
    
    Fix an integer $k$ and consider all possible $f$ and $h$ such that $\abs{f_1(n)} + \abs{g_1(n)} = k$. There are $\binom{2n+2}{k}$ possibilities for $f_1(n)$  and $g_1(n)$: each has $n+1$ coefficients, and we distribute $k$ nonzero coefficients among them. Additionally, for a given choice of $f_1(n)$ and $g_1(n)$, there are $(b-1)^k$ polynomials $f(n)$ and $g(n)$ since each 1 coefficient of $f_1$ or $g_1$ can correspond to any value in $\{1, \dots, b-1\}$. Thus, for a given $k$, there are at most $(b-1)^k\binom{2n+2}{k}$ possibilities for $f(n) \otimes g(n)$. Therefore, $T^b_2(n)$ is a subset of a union of  at most $\sum_{0\leq k\leq\frac{n}{5}}(b-1)^k\binom{2n+2}{k}$ intervals, each of length $b^{-n}$.
    
    We then compute
    \begin{align*}
        m(T^b_2(n)) & \ \leq\ \frac{1}{b^n}\sum_{0\leq k\leq\frac{n}{5}}(b-1)^k\binom{2n+2}{k} \\ &\ \le \  \frac{n}{5b^n}(b-1)^{n/5}\binom{2n+2}{\floor{n/5}} \\
        &\ \le \  \frac{n}{b^n}(b-1)^{n/5}\binom{2n}{\floor{n/5}} \\
        &\ \le \  \frac{n}{b^n}(b-1)^{n/5}\left(\frac{2ne}{n/5}\right)^{n/5} \\
        &\ \le \  \frac{n}{b^n}(10e(b-1))^{n/5} \\
        &\ \le \  n\left(\frac{1.94(b-1)^{1/5}}{b}\right)^n.
    \end{align*}
    Notice that $\frac{1.94(b-1)^{1/5}}{b} < 1$ for $b \geq 2$. Hence, the sum $\sum_{n=1}^\infty m(T^b_2(n))$ converges, so $m(T^b_2) = 0$.
    
\end{proof}

\begin{lemma} \label{lem: t3}
    We have $m(T^b_3) = 0$.
\end{lemma}

\begin{proof}
As in the case of $T^b_1$, we will show that no element of $T^b_3$ is normal, from which the result will follow. 

Without loss of generality, we know that $\liminf_{n \to \infty}\frac{\abs{f_1(n)}}{n} \geq \frac{1}{10}$. Let $k$ be a positive integer such that
\begin{equation*}
   \left(\frac{b-1}{b}\right)^k\ <\ \frac{1}{10}.
\end{equation*} 
Pick a positive integer $r$ such that $|g_1(r-1)| = k$. This is equivalent to choosing $r$ such that $\rho_b(g_1(r-1))$ has exactly $k$ ones. Let $Z$ denote the set of degree $r-1$ polynomials in $\sigma\in \mathcal B_b[x]$ such that $\sigma_1 \oplus g_1 = \sigma_1$. In other words, $Z$ is the set of degree $r-1$ polynomials of $\sigma\in \mathcal B_b[x]$ such that $\beta_i\neq 0\implies \delta_i \neq 0$. We can compute $\abs{Z}$ using a counting argument: If $\beta_i \neq 0$, then $\delta_i \in \{1,\dots, b-1\}$, otherwise $\delta_i\in \{0, \dots, b-1\}$. As $|g_1| = k$ and $\sigma$ has $r$ coefficients, there are  $(b-1)^kb^{r-k}$ possible choices for $\sigma$. 

If $\rho_b(h)$ is normal, we expect the digit strings in $\rho_b(Z)$ to occur at a frequency of $\frac{(b-1)^kb^{r-k}}{b^r} = \left(\frac{b-1}{b}\right)^k$ in $\rho_b(h)$. We show that they instead occur at a frequency of at least $\frac{1}{10}$, from which it follows that $\rho_b(h)$ is not normal. 

Suppose $\alpha'_s = 1$. Then from \Cref{lemma:summand-inequality}, it follows that:
\[
\left(\gamma'_s x^s \oplus \cdots \oplus \gamma'_{s+r-1}x^{s+r-1}\right) \oplus \left(\alpha'_s x^s \otimes g_1(r-1)\right) = \left(\alpha'_s x^s \otimes g_1(r-1)\right) .
\]
Thus $\gamma_s \oplus \cdots \oplus \gamma_{s+r-1}x^{r-1} \in Z$. This observation allows us to lower-bound the frequency of these strings in $\rho_b(h)$:
\begin{align*}
   \frac{1}{10}\ \le \  \liminf_{n \to \infty} \left( \frac{|f_1(n)|}{n} \right) &\ \le \   \liminf_{n \to \infty} \left( \frac{N_{\rho(h)}(n+r-1, Z)}{n} \right) = \liminf_{n \to \infty} \left( \frac{N_{\rho(h)}(n, Z)}{n} \right).
\end{align*}
The above contradicts \Cref{eqn:normality-expression}, thus $h$ is not normal.
\end{proof}

We now prove \Cref{thm:gen-wirsing-measure}, from which \Cref{thm:gen-wirsing} is a corollary.
\begin{proof}
By construction, $\mathcal C^b = T^b_1\cup T^b_2\cup T^b_3$. As a consequence of \Cref{lem: t1}, \Cref{lem: t2}, and \Cref{lem: t3}, we have
\[ m(\mathcal C^b) \ \le \  m(T^b_1) + m(T^b_2) + m(T^b_3) = 0.\]

\end{proof}
\printbibliography
\end{document}

%% file: header/packages.tex
\usepackage{amsthm, amsfonts, times}
\usepackage{epstopdf}

\usepackage[style=alphabetic]{biblatex}
\usepackage{mathrsfs}
\usepackage{centernot}
\usepackage{xfrac}
\usepackage{eufrak}
\usepackage{stmaryrd}
\usepackage{bm}
\usepackage{bbm}
\usepackage[all,cmtip]{xy}
\usepackage{float}
\usepackage[ruled,vlined]{algorithm2e}
\usepackage{tikz-cd}
\usetikzlibrary{positioning,arrows}
\usetikzlibrary{calc}
\usetikzlibrary{matrix}
\usepackage[colorinlistoftodos]{todonotes}

\makeatletter
\providecommand\@dotsep{5}
\def\listtodoname{TODOS: Not Done (Red), Cite (Orange), Blue (Notation)}
\def\listoftodos{\@starttoc{tdo}\listtodoname}
\makeatother

%% file: header/journal-header.tex
\usepackage{amssymb,latexsym}
\usepackage[margin=1in]{geometry}
\usepackage{graphicx}
\usepackage{mathtools}
\usepackage{color}
\usepackage[T1]{fontenc}
\usepackage{verbatim}
\usepackage{xurl}    %does nice formatting of URLs
\usepackage{color}
\usepackage{url}
\usepackage{breakurl}
\usepackage[breaklinks]{hyperref}
\hypersetup{pdfstartview=FitH, pdftitle=Irreducibility over the Max-Min Semiring,colorlinks=true,citecolor=blue,linkcolor=blue,urlcolor=blue}
\usepackage{cleveref}
\usepackage{amsmath}
\usepackage{mathdots}
\usepackage{breqn}
\usepackage{enumitem}
\usepackage{breqn}
\usepackage{fancyhdr}
\newcommand{\bburl}[1]{\textcolor{blue}{\url{#1}}}

\theoremstyle{plain} %italicizes text
\numberwithin{equation}{section}
\newtheorem{thm}{Theorem}[section]
\newtheorem{theorem}[thm]{Theorem}
\newtheorem{prop}[thm]{Proposition}

\newtheorem{lemma}[thm]{Lemma}
\newtheorem{cor}[thm]{Corollary}

\theoremstyle{definition}
\newtheorem{defn}[thm]{Definition}
\crefalias{defn}{definition}

\newtheorem{conjecture}[thm]{Conjecture}

\theoremstyle{remark}

%% file: header/shortcuts.tex
\newcommand{\floor}[1]{\left\lfloor#1\right\rfloor}

\newcommand{\abs}[1]{\left|#1\right|}

\renewcommand{\bar}{\overline}

\newcommand{\E}{\mathbb E}

\newcommand{\N}{\mathbb N}

\newcommand{\R}{\mathbb R}
\newcommand{\Z}{\mathbb Z}